\documentclass[12pt,a4paper]{amsart}
\usepackage[top=35mm, bottom=30mm, left=30mm, right=30mm]{geometry}
\usepackage{mathptmx}
\usepackage{mathrsfs}
\usepackage{verbatim}
\usepackage{url}
\usepackage[all]{xy}
\usepackage{xcolor}
\usepackage[colorlinks=true,citecolor=blue]{hyperref}
\usepackage{amsmath,amsthm}
\usepackage{amssymb}
\usepackage{enumerate}
\usepackage{graphicx}
\usepackage{tikz}
\usetikzlibrary{arrows}
\usetikzlibrary{patterns}

\newtheorem{thm}{Theorem}[section]
\newtheorem{cor}[thm]{Corollary}
\newtheorem{lem}[thm]{Lemma}
\newtheorem{prop}[thm]{Proposition}

\theoremstyle{definition}

\numberwithin{equation}{section}
\begin{document}


\baselineskip=20pt


\title{The structure of periodic point free distal homeomorphisms on the annulus}

\author[E.~Shi]{Enhui Shi}
\address{School of Mathematics and Sciences, Soochow University, Suzhou, Jiangsu 215006, China}
\email{ehshi@suda.edu.cn}

\author[H.~Xu]{Hui Xu*}
\thanks{*Corresponding author}
\address{ Department of mathematics, Shanghai normal university, Shanghai 200234,
 China}
\email{huixu@shnu.edu.cn}

\author[Z.~Yu]{Ziqi YU}
\address{School of Mathematics and Sciences, Soochow University, Suzhou, Jiangsu 215006, China}
\email{20204207013@stu.suda.edu.cn}

\begin{abstract}
Let $A$ be an annulus in the plane $\mathbb R^2$ and $g:A\rightarrow A$ be a boundary components preserving homeomorphism
which is distal and has no periodic points.
Then there is a continuous decomposition of $A$ into $g$-invariant circles
 such that all the restrictions of $g$ on them share a common  irrational rotation number and all these circles are
linearly ordered by the inclusion relation on the sets of bounded components of their complements in $\mathbb R^2$.
\end{abstract}

\keywords{}
\subjclass[2010]{}

\maketitle

\pagestyle{myheadings} \markboth{E. Shi, H. Xu, and Z. Yu}{The structure of periodic point free distal homeomorphisms}

\section{Introduction}

Recurrence is one of the most fundamental notions in the theory of dynamical system. There are various definitions to describe the recurrence behaviors of a point in a system, such as periodic point, almost periodic point, distal point,  recurrent point, regularly recurrent point,
and so on. There has been a considerable progress in studying the structures of the dynamical systems all points of which possess
some kind of recurrence.
\medskip

Montgomery \cite{M37} proved that every pointwise periodic homeomorphism on a connected
 manifold is periodic. For an infinite compact minimal metric system each point of which is regularly recurrent, Block and Keesling \cite{BK04}
 proved that it is topologically conjugate to an adding machine.  Shi, Xu, and Yu \cite{SXY23}
 showed that every pointwise recurrent expansive homeomorphism is topologically conjugate to a subshift of some symbol system, which  extends
 a classical result of Ma\~n\'e \cite{M79} for minimal expansive homeomorphisms. The structure of
  pointwise recurrent maps having the pseudo orbit tracing property is completely determined by Mai and Ye \cite{MY02}.
\medskip

There are also many interesting results around the structures of recurrent maps on low-dimensional spaces.
Mai \cite{Mai05} showed that a pointwise recurrent graph map is either topologically conjugate to an irrational rotation on the circle
 or of finite order.  Naghmouchi \cite{Na13} and Blokh \cite{Bl15} characterized  the structures of pointwise recurrent maps on
 uniquely arcwise connected curves. Kolev and P\'erou\`eme  \cite{KP98} showed recurrent
 homeomorphisms on compact surfaces with negative Euler characteristic are of finite order.
 Foland  \cite{Fo65} proved that any equicontinuous homeomorphism on a closed 2-cell is topologically
 conjugate either to  a reflection of a disk in a diameter or to a rotation of a disk about its center.
 Ritter  \cite{R78} further determined the structure of equicontinuous homeomorphisms on the
 2-sphere and annulus.  Oversteegen and Tymchatyn  \cite{OT90} proved that recurrent homeomorphisms on the plane are periodic.
\medskip

 The notion of distality was introduced by Hilbert for better understanding equicontinu
ity \cite{El58}. The study of minimal distal systems culminates in the beautiful structure theorem
 of Furstenberg \cite{Fu63}, which describes completely the relations between distality and
 equicontinuity for minimal systems. Considering minimal distal actions on compact man
ifolds, Rees \cite{Rees77} proved a sharpening of Furstenberg's structure theorem.

\medskip The aim of the paper is to study the structure of distal homeomorphisms on annulus without periodic
points. One may consult \cite{BC04, Botelho88, Fr88, Handel90} for many interesting related investigations.
\medskip

We obtain the following theorem.

\begin{thm}\label{main}
Let $A$ be an annulus in the plane $\mathbb R^2$ and $g:A\rightarrow A$ be a boundary components preserving homeomorphism
which is distal and has no periodic points.
Then there is a continuous decomposition of $A$ into $g$-invariant circles
 such that all the restrictions of $g$ on them share a common  irrational rotation number and all these circles are
linearly ordered by the inclusion relation on the sets of bounded components of their complements in $\mathbb R^2$.
\end{thm}

The paper is organised as follows. In Section 2, we will introduce some concepts and facts in the theories of
dynamical system and topology. Specially, we will give the definitions of solenoid and adding machine from
the viewpoint of topological groups and recall some results around distal homeomorphisms. In Section 3, we will
show that there exists no adding machine contained in the boundary of an $f$-invariant open disk under some
appropriate assumptions. Based on this result, we show in Section 4 the existence of an $f$-invariant circle
in the boundary just mentioned. In Section 5, we show further that there are sufficiently many $f$-invariant
circles in the annulus. Relying on all these results, we show in Section 6 the existence of the expected
decompositions.

\section{Preliminaries}
In this section, we will recall some notions, notations, and elementary facts in the theories of
dynamical system and topology.

\subsection{Recurrence, minimal sets, and factors}

By a {\it dynamical system} we mean a pair $(X, f)$, where $X$ is a metric space
and $f:X\rightarrow X$ is a homeomorphism. For $x\in X$, the {\it orbit} of $x$ is
the set $O(x, f)\equiv\{f^i(x): i\in\mathbb Z\}$. If there is some $n> 0$ such that
$f^n(x)=x$, then $x$ is called a {\it periodic point} of $f$ and the minimal such
$n$ is called the {\it period} of $x$. A periodic point $x$ of period $1$ is called
a {\it fixed point}, that is $f(x)=x$. A subset $A$ of $\mathbb Z$ is
{\it syndetic} if there is $l>0$ with $A\cap \{p, p+1, \ldots, p+l\}\not=\emptyset$ for any $p\in\mathbb Z$.
We call $x$ an {\it almost periodic point} if for any open neighborhood $U$ of $x$,
the set $N(x, U)\equiv\{i:f^i(x)\in U\}$ is syndetic. If there is a sequence of positive integers
$n_1<n_2<\cdots$ such that $f^{n_i}(x)\rightarrow x$, then we call $x$ a {\it recurrent point};
and if for any open neighborhood $U$ of $x$, there always exists $n>0$ such that $f^n(U)\cap U\not=\emptyset$,
then we call $x$ a {\it nonwandering point}. Clearly, an alomost periodic point is recurrent and a recurrent point is nonwandering.
If each point of $X$ is nonwandering, then we call $f$ {\it nonwandering}.
\medskip

 A subset $S$ of $X$ is  {\it $f$-invariant}
if $f(S)=S$; we use $f|_S$ to denote the restriction of $f$ to $S$.
If $S$ is an $f$-invariant nonempty closed subset of $X$ and contains no proper
$f$-invariant closed subset, then we call $S$ a {\it minimal set} of $f$. If $X$ is a minimal set,
we call the system $(X, f)$  is {\it minimal}. It is clear
from the definition that $S$ is minimal if and only if for each $x\in S$, $O(x, f)$
is dense in $S$.  By an argument of Zorn's lemma, we have that if $X$ is compact, then there always exists a minimal set
of $f$. We have known that each point of a compact minimal set is almost periodic (see e.g. \cite[Chap.1-Theorem 1]{A88}).
\medskip

For any two dynamical systems $(X, f)$ and $(Y, g)$, if there is a continuous surjection $\phi:X\rightarrow Y$
such that $\phi\circ f=g\circ\phi$, then we say that $(Y, g)$ is a {\it factor} of $(X, f)$ and
$(X, f)$ is an {\it extension} of $(Y, g)$; we call $\phi$ a {\it factor map} or a {\it semiconjugation}
between $(X, f)$ and $(Y, g)$; if $\phi$ is a homeomorphism, then we call $(X, f)$ and $(Y, g)$ are
{\it topologically conjugate}. Clearly, if $M$ is a minimal set of $f$, then $\phi(M)$ is a minimal set of $g$.
It is well known that if $\mathbb S^1$ is the unit circle and
$f:\mathbb S^1\rightarrow \mathbb S^1$ is an orientation preserving homeomorphism without periodic points,
then $(\mathbb S^1, f)$ is semiconjugate  to a rigid minimal rotation on $\mathbb S^1$ (see e.g. \cite[Theorem 6.18]{W82}).
\medskip

A topological space $U$ is called an {\it open disk} if it is homeomorphic to the unit open disk in the plane $\mathbb R^2$.
By Riemann mapping theorem, we know that  an open subset $U$ of $\mathbb C$ is an open disk if and only if it is simply connected.

\medskip The following theorem is implied by Brouwer' lemma (see e.g. \cite{Fa87, L06}).

\begin{thm}\label{brouwer}
If $U$ is an open disk and $f:U\rightarrow U$ is an orientation-preserving nonwandering homeomorphism, then $f$ has a fixed point in $U$.
\end{thm}

\subsection{Continuous and semi-continuous decompositions}

Let $(X, \mathcal T)$ be a topological space. A {\it partition} of $X$ is  a collection $\mathcal D$  of
nonempty, mutually disjoint subsets of $X$ such that $\cup\mathcal D=X$. Define $\pi:X\rightarrow \mathcal D$
by letting $\pi(x)$ be the unique $D\in\mathcal D$ such that $x\in D$ for each $x\in X$. We endow
$\mathcal D$ with the largest topology so that $\pi$ is continuous, that is $\mathcal U\subset\mathcal D$ is open
iff $\cup \mathcal U\in\mathcal T$. The topological space $\mathcal D$ so defined is called the {\it decomposition}
of $X$. We also call $\mathcal D$ the {\it quotient space} of $X$ by identifying each element of $\mathcal D$ into a point
and call $\pi$ the {\it quotient map}.
The partition $\mathcal D$ is called {\it upper semi-continuous} provided that whenever $D\in\mathcal D$,
$U\in \mathcal T$, and $D\subset U$, there exists $V\in\mathcal T$ with $D\subset V$ such that if $A\in\mathcal D$
and $A\cap V\not=\emptyset$, then $A\subset U$.
\medskip

Now suppose $X$ is a compact metric space with metric $d$. Let $2^X$ be the collection of all nonempty closed subsets
of $X$ and let $C(X)=\{A\in 2^X: A \ \mbox{is connected}\}$. The {\it Hausdorff metric} $H_d$ on $2^X$
is defined by $H_d(A, B)=\inf\{\epsilon: A\subset B_d(B, \epsilon)\ \mbox{and}\ B\subset B_d(A, \epsilon)\}$
for each $A, B\in 2^X$. Then $2^X$ and $C(X)$ are both compact metric spaces with respect to $H_d$, called
the {\it hyperspaces} of $X$ (see e.g. \cite[Theorems 4.13 and 4.17]{N92}).
Let $\mathcal D$ be a partition of $X$ such that each element of $\mathcal D$ is closed. The partition
$\mathcal D$ is called {\it continuous} if the quotient map $\pi:X\rightarrow \mathcal D$,
thought of as a map from $X$ into $2^X$, is continuous.
\medskip

If $\mathcal D$ is a partition of $X$, then it induces an equivalence relation $R\subset X\times X$ by
defining $(x, y)\in R$ if $\{x, y\}\subset A$ for some $A\in\mathcal D$. The relation $R$ is called
a {\it closed relation} if it is a closed subset of $X\times X$.
\medskip

The following proposition can be seen in \cite[Proposition 2.2]{HJ19}.

\begin{prop}\label{closed-relation}
Let $X$ be a compact metric space. Let $\mathcal D$ be a partition of $X$
and $R$ be the equivalence relation induced by $\mathcal D$. If $R$ is closed, then $\mathcal D$ is upper semi-continuous.
\end{prop}

Proposition \ref{closed-relation} together with the compactness of hyperspaces implies the following proposition.

\begin{prop}\label{semi-con}
Let $f:\mathbb R^2\rightarrow \mathbb R^2$ be a homeomorphism and $M\subset \mathbb R^2$
be a compact minimal set of $f$. Let $\mathcal M$ be the set of all components of $M$ and let
$\mathcal M'=\{\{x\}: x\in \mathbb R^2\setminus M\}$. Then $\mathcal M\cup\mathcal M'$
is an upper semi-continuous decomposition of $\mathbb R^2$.
\end{prop}

A {\it continuum} is a connected compact metric space. If $X$ is a continuum contained in the plane
$\mathbb R^2$ such that $\mathbb R^2\setminus X$ is connected, then we call that $X$
{\it does not separate the plane}. The following theorem is due to R. L. Moore\
(see e.g. \cite[p.533, Theorem 8]{K66} for a slightly more general form).

\begin{thm}\label{moore}
The space of a upper semi-continuous decomposition of $\mathbb R^2$ into continua, which do not separate  $\mathbb R^2$,
is homeomorphic to $\mathbb R^2$.
\end{thm}

\subsection{Solenoids and adding machines}

Let $X$ be a compact metric space with metric $d$ and $f:X\rightarrow X$
 be a homeomorphism. We say $(X, f)$ is {\it equicontinuous} if for each $\epsilon>0$,
 there is a $\delta>0$ such that $d(f^i(x), f^i(y))<\epsilon$ for any $i\in\mathbb Z$, whenever
 $d(x, y)<\delta$. Let $K$ be a compact abelian metric group and $a\in K$. The {\it rotation} $\rho_a:K\rightarrow K$ is defined by
 $\rho_a(x)=ax$ for any $x\in K$. Clearly, if $\{a^n:n\in \mathbb Z\}$ is dense in $K$, then
 $(K, \rho_a)$ is minimal and equicontinuous.
\medskip

 The following theorem shows that minimal rotations on compact abelian metric groups are the only
 equicontinuous minimal systems (see \cite[Theorem 5.18]{W82}).

 \begin{thm}[Halmos-von Neumann]\label{H-N}
 Let $X$ be a compact metric space and $f:X\rightarrow X$ be a minimal and equicontinuous homeomorphism. Then
 $(X, f)$ is topologically conjugate to a minimal rotation on a compact abelian metric group.
 \end{thm}

 For each positive integer $i$, let $K_i$ be a compact metric group and let $f_i:K_{i+1}\rightarrow K_i$
 be a surjective continuous group homomorphism. The {\it inverse limit} of $\{K_i, f_i\}$ is
 $$
 \lim\limits_{\leftarrow}\{K_i, f_i\}\equiv\{(x_i)\in\prod_{i=1}^\infty K_i: x_i=f_i(x_{i+1})\},
 $$
 which is a compact metric group under the multiplication $``\cdot"$ defined by $(x_i)\cdot (y_i)=(x_iy_i)$.
 If each $K_i$ is the unit circle $\{z\in\mathbb C: |z|=1\}$ and $\lim\limits_{\leftarrow}\{K_i, f_i\}$
 is not the circle, then we call $\lim\limits_{\leftarrow}\{K_i, f_i\}$
  a {\it solenoid}; if each $K_i$ is a finite cyclic group and $\lim\limits_{\leftarrow}\{K_i, f_i\}$
 is not finite, then we call  $\lim\limits_{\leftarrow}\{K_i, f_i\}$ an {\it adding machine}. As topological
 spaces,  a solenoid is a homogeneous indecomposable circle-like continuum and an adding machine is
 a Cantor set. We also call a minimal rotation on an adding machine an adding machine.
 \medskip

Since every compact metric group is an inverse limit of compact Lie groups (see e.g. \cite[Chap. 4.6-4.7]{MZ55}), and
the only connected Lie groups of dimension $1$ is the circle group and the only
Lie groups of dimension $0$ are finite groups, the following proposition is clear.

 \begin{prop}\label{mono}
 If $K$ is a connected compact metric group of dimension $1$, then it is either a circle or a solenoid;
 if $K$ is a  compact metric group of dimension $0$ and has a dense cyclic subgroup, then it is either
 a finite cyclic group or an adding machine.
 \end{prop}

A {\it curve} is an $1$-dimensional continuum. The following corollaries are immediate from Theorem \ref{H-N} and Proposition \ref{mono}.

 \begin{cor}\label{1-dim}
 Let $X$ be a curve and $f:X\rightarrow X$ be a minimal equicontinuous homeomorphism. Then $(X, f)$ is
 topologically conjugate to a minimal rotation either on the circle or on a solenoid.
 \end{cor}

  \begin{cor}\label{0-dim}
 Let $X$ be a compact metric space of dimension $0$ and $f:X\rightarrow X$ be a minimal equicontinuous homeomorphism. Then $(X, f)$ is
 either a periodic orbit or topologically conjugate to an adding machine.
 \end{cor}

  The following proposition is shown by Bing \cite{B60}, which is also implied by the main result in \cite{HO16}.
 \begin{prop}\label{planar}
 Solenoids are not planar continua.
 \end{prop}

We call $x$ in a system $(X, f)$ {\it regularly recurrent} if for any open neighborhood $U$ of $x$, there is a positive integer $n$
such that $f^{kn}(x)\in U$ for each $k=0, 1, \cdots$.
\medskip

 The following proposition is implied by the definition (see \cite{BK04} for a characterization of adding machine using regular recurrence).

 \begin{prop}\label{regular-recurrent}
If $(X, f)$ is an adding machine, then each point $x$ of $X$ is regularly recurrent.
 \end{prop}

 \subsection{Structures of distal homeomorphisms}

 Let $X$ be a compact metric space with metric $d$ and let $f:X\rightarrow X$ be a  homeomorphism.
 We call that $(X, f)$ is {\it distal} if for any $x\not=y\in X$, $\inf_{i\in\mathbb Z}\{d(f^i(x), f^i(y))\}>0$ .
We suggest the readers to consult \cite{A88} for the proofs of the following well known facts:
(1)  If $(X,f)$ is distal, then $X$ is a disjoint union of minimal sets; (2) If $(X,f)$ is  minimal and distal and $(Y,g)$ is a factor of $(X,f)$, then  $(Y,g)$ is also minimal and distal; (3) Let $(X,f)$ be a minimal distal system. Then it has a maximal
equicontinuous factor.

 \begin{lem}\cite[\S 6]{Rees77}\label{dim in distal system}
 Let $(X,f)$ be a minimal distal system and $\pi: (X,f)\rightarrow (Y,g)$ be a factor map. Then the covering dimension of the fibers $\pi^{-1}(y), y\in Y$,  is constant and
 \[ \dim(X)=\dim(Y)+\dim \pi^{-1}(y).\]
 \end{lem}


\begin{lem}\cite[p.192, Theorem 3$\cdot$17.13]{Bron79}\label{0-dim fiber}
Let $(X,f)$ be a minimal distal system and $\pi: (X,f)\rightarrow (Y,g)$ be a factor map.  If $(Y,g)$ is equicontinuous and there is some $y\in Y$ with $\dim\pi^{-1}(y)=0$, then $(X, f)$ is also equicontinuous.
\end{lem}

Clearly, Lemma \ref{0-dim fiber} implies a distal minimal system of zero dimension is equicontinuous. In fact,
this is also true for non-minimal distal systems of zero dimension (see \cite[Corollary 1.9]{AGW07}).

\begin{prop}\label{dis=equ}
Let $X$ be a compact connected metric space and $f:X\rightarrow X$ be a minimal distal homeomorphism.
If $\dim(X)= 1$, then $(X, f)$ is equicontinuous.
\end{prop}

\begin{proof}
Suppose that $(X, f)$ is not equicontinuous. Then  the maximal equicontinuous factor of $(X,f)$ is  nontrivial. Let $\pi: (X, f)\rightarrow (Y, g)$ be the factor map to its maximal equicontinuous factor. In particular, $Y$ is connected. By Lemma \ref{dim in distal system}, we have $\dim(Y)=1$ and for each $y\in Y$, $\dim\pi^{-1}(y)=0$. Then it follows from Lemma \ref{0-dim fiber} that $(X, f)$ is equicontinuous. This contradiction shows that $(X, f)$ is equicontinuous.
\end{proof}

 \section{Nonexistence of an adding machine in the boundary of an open disk}

We call a topological space $X$ is an {\it arc} (resp. {\it open arc}) if it is homeomorphic to the
closed interval $[0, 1]$ (resp. the open interval (0, 1)). We call $X$ a {\it circle} if it is homeomorphic
to the unit circle in the plane, that is $X$ is a simple closed curve.
\medskip

Let $f:\mathbb R^2\rightarrow \mathbb R^2$ be an orientation preserving homeomorphism and
$U\subset\mathbb R^2$ be a bounded $f$-invariant open disk. A {\it cross-cut} of
$U$ is an open arc $\gamma$ in $U$ with $\overline\gamma$ being an arc jointing two points
of $\partial U$. A {\it cross-section} of $U$ is any connected component of $U\setminus\gamma$
for some cross cut $\gamma$ of $U$. A {\it chain} for $U$ is a sequence of sections
$\mathcal C=(D_i)_{i=1}^\infty$ such that $D_1\supset D_2\supset\cdots$ and $\overline{\partial_U D_i}\cap \overline{\partial_U D_i}=\emptyset$
for all $i\not=j$. Two chains $(D_i)_{i=1}^\infty$ and $(D_i')_{i=1}^\infty$ are called {\it equivalent} if for any $i>0$ there is $j>i$ such that
$D_j\subset D_i'$ and $D_j'\subset D_i$. A chain $(D_i)_{i=1}^\infty$ is called a {\it prime chain} if ${\rm diam}(\partial_U D_i)\rightarrow 0$.
An equivalence class of prime chains is called a {\it prime end} of $U$. We use $ b_{\mathcal E}(U)$ to denote  the set of all
prime ends of $U$. Let $\hat{U}=U\cup b_{\mathcal E}(U)$.
\medskip

 Now we topologize $\hat{U}$ as follows.
For a cross-section $D$ of $U$ and for a prime chain $(U_i)$ representing $p\in b_{\mathcal E}(U)$, if
$U_i\subset D$ for sufficiently large $i$, then we call $p$ {\it divides} $D$.
Set ${\mathcal E}(D)=\{p\in b_{\mathcal E}(U): p\ \mbox{divides}\ D\}$.
Consider the family ${\mathcal B}$ consisting of all sets of the form $D\cup {\mathcal E}(D)$ for some
cross-section $D$, together with all open subsets of $U$. Then $\mathcal B$ is a topological basis on
$\hat{U}$. We  endow $\hat{U}$ with the topology generated by $\mathcal B$.
\medskip

The following theorem is known as the Carath\'eodory's prime ends compactification theorem (see \cite{C13, Ca13}).

\begin{thm}[Prime ends compactification]
$\hat{U}$ is homeomorphic to the unit closed disk and $b_{\mathcal E}(U)$ is homeomorphic to
the unite circle $\mathbb S^1$.
\end{thm}

It is well known that the homeomorphism $f|_U$ can be extended to a homeomorphism
$\hat{f}:\hat{U}\rightarrow \hat{U}$. We call the rotation number of $\hat{f}|_{\mathbb S^1}$
the {\it prime ends rotation number} of $f|_{\overline U}$.
\medskip

The following theorem is due to Cartwright and Littlewood \cite{CL51}. One may consult \cite{KLN15}
for the proof of the converse direction under more general settings.

\begin{thm}\label{p-e-r}
If $f$ is nonwandering and has no periodic point in $\partial U$, then the prime ends rotation number
of $f|_{\overline U}$ is irrational.
\end{thm}

Now we use Theorem \ref{p-e-r} to prove a key result.

\begin{prop}\label{no-add}
If the prime ends rotation number of $f|_{\overline U}$ is irrational, then no minimal set in
$\partial U$ is an adding machine.
\end{prop}

\begin{proof}
Assume to the contrary that there is a minimal set $K\subset \partial U$, which is an adding machine.
Fix $p\in K$. Since the rotation number of $\hat{f}|_{b_{\mathcal E}(U)}$
is irrational, $\hat{f}|_{b_{\mathcal E}(U)}$ is semi-conjugate to an irrational rotation on the unit circle.
Then we take a cross-cut $\gamma$ of $U$ such that for each cross-section $D$ of $\gamma$,
${\mathcal E}(D)$ contains the closure of a wandering interval (if any) of $\hat{f}|_{b_{\mathcal E}(U)}$
and such that $p$ is not an endpoint of $\gamma$ in ${\overline U}$.
Take a sufficiently small  $\epsilon>0$ such that $V\equiv B(p, \epsilon)\cap U$ is contained in a cross-section $D$ of $\gamma$.
Let $D'$ be the cross-section of $\gamma$ other than $D$. By Proposition \ref{regular-recurrent},
there is some positive integer $n$ such that
\begin{equation}\label{p-1}
f^{kn}(p)\in B(p, \epsilon)
\end{equation}
 for all $k\geq 0$.
Take a sequence $(x_i)$ in $V$ such that $x_i\rightarrow p$. Passing to a subsequence if necessary,
we suppose $x_i\rightarrow q\in b_{\mathcal E}(U)\subset\hat{U}$.
So, there is some $l>0$ such that
\begin{equation}\label{p-2}
\hat{f}^{ln}(q)\in {\mathcal E}(D').
\end{equation}
Then for sufficiently large $i$, by equations (\ref{p-1}) and  (\ref{p-2}), and by the continuity, we have both
$f^{ln}(x_i)\in V\subset D$ and $f^{ln}(x_i)=\hat{f}^{ln}(x_i)\in D'$ (see Figure 1 and Figure 2). This is a contradiction.
\end{proof}

\begin{figure}
\begin{tikzpicture}
\draw  [rounded corners=10pt] (0,6) -- (0,0) -- (10,0) -- (10,6)..controls (8.5,6) .. (8.5,5) ;
\draw  plot [smooth, tension=2] coordinates {(8.5,5) (8,2) (7,6) (6,2) (5,6) (4,2) (3.8,5.7) };
\draw [rounded corners=10pt] (3.8,5.8)--(3.8,6)--(3,6)--(3,5);
\draw  plot [smooth, tension=2] coordinates {(3,5)(2.5,2) (2,6) (1.5,2)  (1.2,6) (1,2.2) };
\draw [rounded corners=5pt] (1,2.2)--(1,2)--(0.5,2)--(0.5,2.5);
\draw [loosely dotted, thick] (0.5,2.5)--(0.5, 3.5);
\draw [blue, rounded corners=10pt] (-1,3.5) rectangle (3.5,5.5);
\draw plot [smooth, tension=2] coordinates {(6,0) (8,1.5) (9,1) (10,2)} node [left=3cm,below=0.5cm] {$\gamma$};
\draw (8,0.5) node  {$D'$};
\draw (4,1) node  {$D$};
\draw (5, 0) node[below]  {$U$};
 \filldraw [red] (0,4.4) circle (2pt) node[left, black] {$f^{ln}(p)$}(0,5) circle (2pt) node[black,left] {$p$} ;
  \filldraw [red] (3.3,5) circle (2pt) node[below, black] {$x_1$} (2,5) circle (2pt) node[black,below] {$x_2$}  (1.15,5) circle (2pt) node[black,below] {$x_3$};
  \filldraw [draw=none,pattern=north east lines] (1,3.5)--(1.28,3.5)--(1.29,3.7)--(1.3,5.5)    --(1.1,5.5)--(1.05,5.3)--(1,4.8)--(1,5);
   \filldraw [draw=none,pattern=north east lines] (1.7,3.5)--(2.25,3.5)--(2.25,5)--(2.22,5.5)    --(1.8,5.5)--(1.72,5);
     \filldraw [draw=none,pattern=north east lines] (3,3.5)--(3.3,3.5)--(3.5,3.8)--(3.5,5.2)--(3.4,5.4)--(3.35,5.4)--(3.3,5.5)    --(3,5.5)--(3,5);
\end{tikzpicture}
\caption{ }
\end{figure}

\begin{figure}
\begin{tikzpicture}
\draw (3,3) circle (3cm) node [left=1cm] {$D$} node[below=3cm]{$\hat{U}$};
\draw plot [smooth, tension=2] coordinates {(2,.15)(3.5,2.5)(6,2.9)} node [left=2cm] {$\gamma$} node[left=1cm, below=0.5cm] {$D'$} ;
\draw  (0.5,4.7).. controls +(right:0.5cm) and +(down:0.5cm)..(1.2,5.4);
\draw  (0.3,4.3).. controls +(right:0.5cm) and +(down:1.5cm)..(1.7,5.7) node[below=1.7cm,left=1cm] {$\vdots$};
\draw  (3,6).. controls +(left:0.3cm) and +(down:0.5cm)..(3.5,5.95);
\draw  (2.5,5.95).. controls +(left:0.5cm) and +(down:1.5cm)..(3.8,5.9);
\draw  (4.5,5.6).. controls +(left:1cm) and +(down:1cm)..(5.4,4.8);
 \filldraw [draw=none,pattern=north east lines]  (0.3,4.3)--(0.5,4.3)--(0.6,4.3)--(0.8,4.32)--(1.1,4.38)--(1.2,4.46)--(1.4,4.65)--(1.5,4.73)--(1.55,4.8)--(1.6,5)--(1.65,5.2)--(1.7,5.6)--(1.7,5.65)--(1.7,5.7)--(1.2,5.4)--(1.15,5)--(1,4.81)--(0.9,4.8)--(.7,4.7)--(0.51,4.7);
  \filldraw [draw=none,pattern=north east lines]   (3,6)--(2.5,5.95)--(2.46,5.8)--(2.6,5.58)--(2.8,5.47)--(3,5.35)--(3.5,5.26)--(3.6,5.3)--(3.75,5.5)--(3.8,5.9)--(3.5,5.95)--(3.5,5.8)--(3.4,5.7)--(3.3,5.71)--(3.1,5.8)--(2.9,5.9);
\filldraw [draw=none,pattern=north east lines]   (4.5,5.6)--(4.2,5.5)--(4.3, 5.1)--(4.4,4.99)--(4.45,4.92)--(4.5,4.95)--(5,4.5)--(5.1,4.45)--(5.2,4.49)--(5.38,4.52)--(5.4,4.85)--(5,5.2)--(4.8,5.36);
 \filldraw [red] (3.3,5.5) circle (2pt) node[below, black] {$x_2$} (4.8,5) circle (2pt) node[black,below] {$x_1$}  (1.1,4.65) circle (2pt) node[black,right] {$x_3$};
  \filldraw [red] (.1,3.7) circle (2pt) node[left, black] {$q$} ;
   \filldraw [red] (4,0.2) circle (2pt) node[above,black] {$\hat{f}^{ln}(q)$} ;
\end{tikzpicture}
\caption{ }
\end{figure}

 \section{Existence of an $f$-invariant circle in the boundary of an open disk}

 \begin{prop}\label{inv-circ}
 Let $f:\mathbb R^2\rightarrow\mathbb R^2$ be a nonwandering homeomorphism and let
 $U$ be an $f$-invariant open disk. If $f|_{\partial U}:\partial U\rightarrow\partial U$ is distal and
$f$ has no periodic points except for an only fixed point $O\in U$, then there is an $f$-invariant circle $C$ in $\partial U$
 such that $(C, f|_C)$ is minimal and $O$ belongs to the bounded component of $\mathbb R^2\setminus C$.
 \end{prop}

 \begin{proof}
 Fix a minimal set $M$ in $\partial U$. Then ${\rm dim}(M)\leq 1$.
Noting that $f$ is nonwandering and has no periodic point in $\partial U$, by Theorem \ref{p-e-r},
we have the prime end rotation number of $f|_{\overline U}$ is irrational. Then it follows from
Corollary \ref{0-dim} and Proposition \ref{no-add} that $M$ is not an adding machine. So ${\rm dim}(M)=1$.
Thus $M$ has a component $K$ with ${\rm dim}(K)=1$.
\medskip

Now we discuss into several cases:
\medskip

{\bf Case 1.} There is some $n\geq 1$ such that $f^n(K)=K$ and $f^i(K)\cap K=\emptyset$ for $1\leq i<n$.
Clearly, $(K, f^n)$ is minimal. Then, from Proposition \ref{dis=equ}, it is equicontinuous.  By Corollary \ref{1-dim} and Proposition \ref{planar},
$K$ is a circle.
 \medskip

{\bf Subcase 1.1.} $n=1$. Let $C=K$ and let $D$ be the bounded component of $\mathbb R^2\setminus C$.
Since $f$ has no periodic points in $C$, so by Brouwer's fixed point theorem, $O\in D$. Thus $C$ satisfies
the requirement.
 \medskip

{\bf Subcase 1.2.} $n>1$. Let $C_i=f^i(K),\ i=0,\ldots, n-1$. Then $C_i$ are pairwise disjoint. Let $D_i$
be the bounded component of $\mathbb R^2\setminus C_i$. If there are $i\not=j$ such that $C_i\subset D_j$,
then $f^{i-j}({\overline D_j})\subset D_j$. This contradicts the assumption that $f$ is nonwandering.
So these $D_i$ are pairwise disjoint. Since each $D_i$ contains a fixed point of $f^n$ by Brouwer's fixed point theorem,
 this contradicts the assumption that $O$ is the only periodic points of $f$. So this subcase does not occur.
 \medskip

 {\bf Case 2.} $f^i(K),\ i\in\mathbb Z$,\  are pairwise disjoint. Write $K_i=f^i(K)$. If $K$  separates the plane,
 then $\mathbb R^2\setminus K$ has a bounded component, so is each $K_i$. Similar to the arguments in Subcase 1.2,
 we have that for any $i\not=j$, $K_i$ is contained in the unbounded component of $\mathbb R^2\setminus K_j$.
 Thus any bounded component of $\mathbb R^2\setminus K$ is a wandering open set of $f$. This is a contradiction.
  \medskip

 From the above discussions, we get the following claim.
   \medskip

\noindent  {\bf Claim A.} Either the conclusion of Proposition \ref{inv-circ} holds, or $M$ has infinitely many components and any nondegenerate component of $M$ does not separate the plane.
   \medskip

 If the conclusion of  Proposition \ref{inv-circ} does not hold, then by Claim A together with Proposition \ref{semi-con}
 and Theorem \ref{moore},
 we  get a factor $g:\mathbb R^2\rightarrow \mathbb R^2$ of $f$ by identifying each component of $M$ to a point.
Let $\pi:\mathbb R^2\rightarrow \mathbb R^2$  be the factor map. Then $\pi(U)$ is a $g$-invariant open disk and $g$ is a nonwandering
homeomorphism and has no periodic point in $\partial \pi(U)$. So, by Theorem \ref{p-e-r}, the prime ends rotation number of $g|_{\pi(\overline U)}$
is irrational.
Noting that $g|_{\partial \pi(U)}$ is still distal and $\pi(M)$ is totally disconnected and infinite, by Corollary \ref{0-dim},
we see that $(\pi(M), g)$ is an adding machine contained in the boundary of $\pi(U)$.
This contradicts Proposition \ref{no-add}.
   \medskip

All together, we complete the proof.
 \end{proof}

 \section{Existence of an intermediate $f$-invariant circle}\label{ex-circ}

\begin{lem}\label{open-orbit}
Let $X$ be a compact metric space and let $f:X\rightarrow X$ be a distal homeomorphism.
If $K$ is an  $f$-invariant proper closed subset of $X$, then
there are  $\delta>0$ and a nonempty open subset $U$ of $X$ such that
$f^i(U)\cap B(K, \delta)=\emptyset$ for each $i\in \mathbb Z$.
\end{lem}

\begin{proof}
For each positive integer $n$, let $V_n=\{x\in X: f^i(x)\in B(K, \frac{1}{n})\ \mbox{for some}\ i\in \mathbb Z\}$.
If the conclusion of Lemma \ref{open-orbit} does not hold, then $V_n$ is a dense open subset of $X$ for each $n$.
Thus by Baire's Theorem, $G\equiv\cap_{n=0}^\infty V_n$ is a dense $G_\delta$-set. Take $x\in G\setminus K$.
Then $\overline{O(x, f)}\cap K\not=\emptyset$. This contradicts the minimality of $\overline{O(x, f)}$.
\end{proof}

For any two circles $C, C'$ in the plane $\mathbb R^2$, write $C\prec C'$ if $C$
is contained in the bounded component of $\mathbb R^2\setminus C'$.

\begin{prop}\label{inter}
Let $f:\mathbb R^2\rightarrow \mathbb R^2$ be an orientation preserving nonwandering homeomorphism.
Let $A$ be an $f$-invariant annulus with two boundary circles $C_1\prec C_2$. Suppose
$f|_A$ is distal and $f$ has no periodic points except for an only fixed point $O$ in the bounded component of $\mathbb R^2\setminus C_1$.
Then there is an $f$-invariant circle $C$ with $C_1\prec C\prec C_2$.
\end{prop}

\begin{proof}
By Lemma \ref{open-orbit}, we can take $\delta >0$ and a nonempty open set $U$ in $A$ such that
$f^i(U)\cap B(C_1\cup C_2, \delta)=\emptyset$ for each $i\in \mathbb Z$. Set $W=\cup_{i\in\mathbb Z} f^i(U)$.
Let $K$ be the unbounded component of $\mathbb R^2\setminus W$. Then $K$ is $f$-invariant and
$C_2\subset \stackrel{\circ}{K}$. Let $V$ be a component of $\mathbb R^2\setminus K$. Then $V$ is an open disk by a direct application
of Jordan separation theorem. Since  $f$ is nonwandering, there is some $n\geq 0$ with $f^n(V)=V$.
Then, by Theorem \ref{brouwer}, there is a periodic point of $f^n$
in $V$, so is for $f$. Thus by the assumption, we have $O\in V$. This implies $C_1\subset V$. From the above
discussions, we see that $V$ is the only component of $\mathbb R^2\setminus K$, and hence $f(V)=V$.
Now applying Proposition \ref{inv-circ}, we get the required circle.
\end{proof}

 \section{A decomposition of the annulus into $f$-invariant circles}

In this section, we will complete the proof of the main Theorem. All assumptions are as in Theorem \ref{main}.
WLOG, we may suppose the annulus $A=\{z\in\mathbb C:1\leq |z|\leq 2\}$. Then we extend $g:A\rightarrow A$ to
$f:\mathbb C\rightarrow \mathbb C$ by defining
$$f(x)=\left\{\begin{array}{llll}\frac{|x|}{2}\cdot g(\frac{2x}{|x|}), & |x|>2;\\
g(x), & 1\leq |x|\leq 2;\\
|x|\cdot g(\frac{x}{|x|}), & 0<|x|<1;\\
0, & |x|=0.\end{array}
\right.$$
It is clear from the definition that $f$ is an orientation preserving nonwandering homeomorphism on the plane and
has no periodic points except for the only fixed point $0$.
\medskip

Write $C_1=\{z: |z|=1\}$ and $C_2=\{z: |z|=2\}$. For each circle $C$ in the plane, we use
$D(C)$ and $OD(C)$ to denote the bounded component  and  unbounded component of $\mathbb R^2\setminus C$, respectively.
Let $\prec$ be the transitive order defined in Section \ref{ex-circ}; that is, for circles $C$ and $C'$ in the plane,
  $C \prec C'$ iff $C\subset D(C')$.
Let $$\mathcal C=\{C: C\ \mbox{is an $f$-invariant circle and}\ C_1\prec C\prec C_2\}\cup\{C_1, C_2\}.$$
Let $\mathcal{T}$ be the family of all chains of $\mathcal C$ respect to $\prec$.
Then $\mathcal T$ is a partial set with respect to the inclusion relation on the power set
of $\mathbb R^2$.
Using Zorn's lemma, there is a maximal chain $\mathcal P$ in $\mathcal{T}$.
\medskip

\noindent{\bf Claim A.} $\mathcal P$ is a partition of $A$.
\begin{proof}[Proof of Claim A]
Assume to the contrary that there is some $v\in A\setminus {\cup \mathcal P}$. Since
$C_1, C_2\in \mathcal P$ by the maximality of $\mathcal P$, we have $v\in \stackrel{\circ}{A}$.
Set ${\mathcal P}_1=\{C\in\mathcal P: v\notin D(C)\}$ and  ${\mathcal P}_2=\{C\in\mathcal P: v\in D(C)\}$.
Then $C_1\in {\mathcal P}_1$ and  $C_2\in {\mathcal P}_2$.
\medskip

Now we discuss into several cases.
\medskip

{\bf Case 1.} ${\mathcal P}_1$ has no maximal element. Let $U=\cup_{C\in {\mathcal P}_1} D(C)$.
Then $U$ is an $f$-invariant open disk. From Proposition \ref{inv-circ},
we have an $f$-invariant circle $C_3$ in $\partial U$ with $0\in D(C_3)$.
Clearly, $C_3\not\in\mathcal P$ and $C\prec C_3$ for any $C\in{\mathcal P}_1$.
If $C_3\prec C$ for any $C\in{\mathcal P}_2$, then $\{C_3\}\cup\mathcal P$ is a chain, which
contradicts the maximality of $\mathcal P$. So,
there must exit a $C_4\in {\mathcal P}_2$ such that $C_3\cap C_4\not=\emptyset$.
Let $V$ be a component of $D(C_4)\setminus \overline {D(C_3)}$. Then
$V$ is a component of $\mathbb R^2\setminus (C_3\cup C_4)$, and hence it is an open disk.
Noting that $f$ is nonwandering, we have $f^n(V)=V$ for some $n\geq 0$.
Then by Theorem \ref{brouwer}, $f$ has a periodic point in $V$. This is a contradiction.
So, Case 1 does not happen.

\begin{figure}
\begin{tikzpicture}
 \draw plot[smooth cycle] coordinates {(1,3.5)(0.8,5) (2,5.5) (2.6,7.5) (4,7)(6,7.2)(5.7,6)(7.5,5)(6, 3)(6,1.5)(4.7,0.8)(3.5,.8)(2.6,1.5) (1.5,1.4)(0.6, 2.2)} node [above=4.5cm, right=5.2cm]{$C_4$};
\draw[blue] plot[smooth cycle] coordinates{(1,3.5)(1.8, 4) (2,5.5) (3,6)(4,7)(5,6) (5.7,6) (5.5, 4)(6,3) (3.8, 2.5)(2.6,1.5) (1.5, 2)}  node [above=0.3cm, right=1.3cm]{$C_3$};

 \filldraw [black] (4,4) circle (2pt) node[below, black] {$O$} ;

\draw (4,4) circle (4.5cm) node [above=4.5cm, right=1.2cm] {$C_2$};
\draw (4,4) circle (1cm)  node [above=0.67cm, right=0.67cm] {$C_1$};
 \filldraw [draw=none,pattern=north east lines]   (1,3.6)--(0.9,4)--(0.85, 4.5)--(0.8,5)--(1.9,5.4)--(1.8,4)--(1.3, 3.7)--(1.1, 3.68)--(1,3.6);
  \filldraw [draw=none,pattern=north east lines]   (2.05,5.6)--(2.6,7.5)--(2.9,7.5)--(3.2,7.37)--(3.9,6.95)--(3.7,6.85)--(3,6)--(2.05,5.6);
  \filldraw [draw=none,pattern=north east lines]   (4.1,7)--(6,7.2)--(6,7)--(5.82, 6.6)--(5.6,6.1)--(5,6)--(4.1,7);
   \filldraw [draw=none, pattern=north east lines]   (5.7,6)--(6,5.8)--(6.8,5.5)--(7.2, 5.3)--(7.5,5)--(7.3, 4.6)--(6.3,3.5)--(6.07,3.15)--(5.5,4)--(5.7,6);
 \filldraw [draw=none,pattern=north east lines]   (5.9,2.9)--(6,1.5)--(5.5, 1.1)--(4.7,0.8)--(4,0.75)--(3.5,0.8)--(2.6,1.5)--(3.8,2.5)--(5.9,2.9);
 \filldraw [draw=none,pattern=north east lines]   (2.2,1.5)--(1.6,1.4)--(1,1.7)--(0.6,2.2)--(1,3.5)--(1.1,2.85)--(1.3, 2.3)--(1.5,2)--(1.8,1.7)--(2.2,1.5);
\end{tikzpicture}
\caption{ }
\end{figure}

\medskip

{\bf Case 2.} ${\mathcal P}_2$ has no minimal element. We consider the Riemann sphere
$\hat {\mathbb C}=\mathbb C\cup \{\infty\}$. Let $\mathbb L=\hat {\mathbb C}\setminus \{0\}$.
Then $\mathbb L$ is a plane. Define $\hat{f}:\mathbb L\rightarrow \mathbb L$ by letting $\hat{f}(\infty)=\infty$ and
$\hat{f}(x)=f(x)$ for any $x\in \mathbb C\setminus \{0\}$. Then  $\hat{f}$ is an orientation preserving nonwandering homeomorphism
on the plane $\mathbb L$ and has no periodic points except for the only fixed point $\infty$.
Similar to the discussions in Case 1, we see that Case 2 does not happen.
\medskip

{\bf Case 3.} ${\mathcal P}_1$ has the maximal element $C_5$ and ${\mathcal P}_2$ has the minimal element $C_6$.
Then $C_5\prec C_6$. Applying Proposition \ref{inter}, we get an $f$-invariant circle $C_7$ such that
$C_5\prec C_7\prec C_6$. Then $\{C_7\}\cup \mathcal P$ is a chain. This contradicts the maximality
of $\mathcal P$. Thus Case 3 does not occur.
\medskip

So $A=\cup\mathcal P$; that is $\mathcal P$ is a partition of $A$.
\end{proof}

\noindent{\bf Claim B}. $\prec$ is a dense and complete linear order on $\mathcal{P}$.

\begin{proof}[Proof of Claim B]
(1) Linearity. For any  distinct $C, C'\in \mathcal{P}$, we have $C\cap C'=\emptyset$ and hence either $C\subset D(C')$ or $C'\subset D(C)$. This shows that either $C\prec C'$ or $C'\prec C$. Thus $\prec$ is a linear order.

(2) Density. Let $C, C'\in \mathcal{P}$ be with $C\prec C'$. Then we have $\overline{D(C)}\subsetneq D(C')$. Now for any $x\in D(C')\setminus \overline{D(C)}$, there is some $C''\in\mathcal{P}$ such that $x\in C''$. It follows from the definition of $\prec$ that $C\prec C''\prec C'$. This shows that $\prec$ is a dense order.

(3) Completeness. To the contrary, assume that $\prec$ is incomplete. That is there is a Dedekind Gap, which means that there is a partition $\mathcal{P}=\mathcal{L}\cup \mathcal{U}$ such that
\begin{itemize}
\item For any $C\in \mathcal{L}$ and $C'\in \mathcal{U}$, $C\prec C'$,
\item $\mathcal{L}$ has a maximal element $C^{*}$ and $\mathcal{U}$ has a minimal element $C_{*}$.
\end{itemize}
 Then we have
 \[\bigcup_{C\in\mathcal{L}} C=\overline{D(C^{*})}\setminus D(C_1) \text{ and } \bigcup_{C\in\mathcal{U}} C=\overline{OD(C_{*})}\setminus OD(C_2),  \]
 both of which are closed in $A$. But this contradicts the connectedness of $A$. This shows that $\prec$ is complete.
\end{proof}

Now Claim B implies that $(\mathcal{P},\prec)$ endowed with the ordering topology is homeomorphism to a closed interval. WLOG, we may assume that $(\mathcal{P},\prec)\cong [1,2]$ and use $C_{r}$ to denote elements of $\mathcal{P}$ with $r\in [1,2]$.
\medskip

Recall that  $2^{A}$ is the hyperspace of $A$ endowed with the Hausdorff metric.
\medskip

\noindent{\bf Claim C}. $\mathcal{P}$ is closed in $2^{A}$. Specially, $\mathcal{P}$ is a continuous decomposition.
\begin{proof}[Proof of Claim C]
Let $C_{r_n}$ be a sequence in $\mathcal{P}$ that converges to $K$ in $2^{A}$ under Hausdorff metric. Since $(\mathcal{P},\prec)\cong [1,2]$, by passing to some subsequence, we may assume that $C_{r_{n}}\overset{\prec}{\longrightarrow} C_{r}$ under the ordering topology for some $r\in [1,2]$.
\medskip

Next we will show that $K=C_{r}$ which implies that $\mathcal{P}$ is closed in $2^{A}$. We may assume as well  $C_{r_n}\prec C_r$ for each $n$. Fix an $x\in K$. Then it follows from the definition of Hausdorff metric that there is a sequence $(x_n)$ with $x_n\in C_{r_n}$ such that $x_n\rightarrow x$. Since $x_n\in C_{r_n}\subset D(C_r)$, we have $x\in \overline{D(C_r)}$. To show $x\in C_r$, we assume that $x\in D(C_r)$. Then there is some $s\in [1,r)$ such that $x\in C_{s}$. Since $r_n\rightarrow r$, we have $s< r_n\leq r$ and hence $C_{s}\prec C_{r_n}\preceq C_r$ for any sufficiently large $n$. But in this case,  $x_n$ cannot converge to $x$; this is a contradiction. To sum up, we have $x\in C_r$. Hence $K\subset C_r$ as $x$ is arbitrary. On the other hand,  for any $y\in C_r$, there is some subsequence $(y_{n_i})$ with $y_{n_i}\in C_{r_{n_i}}$ such that $y_{n_{i}}\rightarrow y$.  Indeed, take any point $z\in C_1$ and let $L$ be the segment connecting $z$ and $y$ in $A$. Then $L\cap C_{r_n}\neq\emptyset$ for each $n$ and we choose some $y_n\in L\cap C_{r_n}$. It is clear that $y$ is a limit point of $(y_n)$. The above arguments show that $K=C_r$  and the closedness of $\mathcal{P}$ in $2^A$ is followed.
\end{proof}
\medskip

For any  orientation preserving homeomorphism $\phi$ on a circle, we use $\rho(\phi)$ to denote
the rotation number of $\phi$.
\medskip

\noindent{\bf Claim D.} The rotation numbers of $f|_C$,\ $C\in\mathcal P$,\ are the same
irrational number.
\medskip
\begin{proof}[Proof of Claim D]
Assume to contrary that there are $C'\not= C''\in\mathcal P$ such that the rotation numbers  $\rho(f|_{C'})\not=\rho(f|_{C''})$.
Let $\tilde A$ be the annulus in $A$ with boundary $C'\cup C''$. Write $\alpha'=\rho(f|_{C'})$ and $\alpha''=\rho(f|_{C''})$.
Take a homeomorphism
$h:\tilde A\rightarrow \mathbb S^1\times [0, 1]$ such that $h(C')=\mathbb S^1\times\{0\}$ and
$h(C'')=\mathbb S^1\times\{1\}$. Let $\tilde f=hf|_Ah^{-1}$. Then $\rho(\tilde f|_{h(C')})=\alpha'$
and $\rho(\tilde f|_{h(C'')})=\alpha''$. Let $\mathbb R\times [0, 1]$ be the universal covering  of
$ \mathbb S^1\times [0, 1]$ and let $F: \mathbb R\times [0, 1]\rightarrow \mathbb R\times [0, 1]$
be a lift of $\tilde f$. Noting that $\tilde f$ is homotopic to the identity,
we have $TF=FT$, where $T$ is the unit translation on $\mathbb R\times [0, 1]$
defined by $T(s, t)=(s+1, t)$. Let $\pi:\mathbb R\times [0, 1]\rightarrow \mathbb R$ be the projection to the first coordinate, that is $\pi(s, t)=s$.
For any map $\psi:\mathbb R\times [0, 1]\rightarrow \mathbb R\times [0, 1]$, write $\psi_1=\pi\psi$. 
Then there are $m', m''\in\mathbb Z$ such that
$$
\lim\limits_{n\rightarrow\infty}\frac{F^n_1(x, 0)-x}{n}=\alpha'+m'
$$
and
$$
\lim\limits_{n\rightarrow\infty}\frac{F^n_1(x, 1)-x}{n}=\alpha''+m''.
$$
WLOG, suppose $\alpha'+m'<\alpha''+m''$. 
Take a rational number $\frac{p}{q}$ with $$\alpha'+m'<\frac{p}{q}<\alpha''+m''.$$ 
 Then we have
$$
\lim\limits_{n\rightarrow\infty}\frac{(T^{-p}F^q)^n_1(x, 0)-x}{n}=q(\alpha'+m')-p<0
$$
and
$$
\lim\limits_{n\rightarrow\infty}\frac{(T^{-p}F^q)^n_1(x, 1)-x}{n}=q(\alpha''+m'')-p>0.
$$
Since $f$ is nonwandering, it follows from \cite[Theorem 3.3]{Fr88} that
$\tilde f^q$ has a fixed point in $ \mathbb S^1\times [0, 1]$. That is $\tilde f$ has a periodic point, so is
$f$. This is a contradiction. So the rotation numbers of $f|_C$,\ $C\in\mathcal P$,\ are the same, the irrationality
of which clearly follows from the Poincar\'e's classification theorem for circle homeomorphisms.
\end{proof}

All together, we complete the proof of Theorem \ref{main}.



\begin{thebibliography}{HD}

\normalsize
\baselineskip=17pt
\bibitem{A88} J. Auslander, {\it Minimal flows and their extensions}, in: North-Holland Mathematics Studies, Vol. 153, in: Notas de Matemtica [Mathematical Notes], vol. 122, North-Holland Publishing Co., Amsterdam, 1988.

\bibitem{AGW07} J. Auslander, E. Glasner and B. Weiss, {\it On recurrence in zero dimensional flows}, Forum Math. 19 (2007),  107-114.

\bibitem{BC04}  F. B\'eguin, S. Crovisier, F. Le Roux, and A. Patou, {\it Pseudo-rotations of the closed annulus: variation on a
 theorem of J Kwapisz},  Nonlinearity 17 (2004) 1427-1453.

\bibitem{B60} R. H. Bing, {\it A simple closed curve is the only homogeneous bounded plane continuum that contains an arc},
Canad. J. Math. 12 (1960), 209-230.

\bibitem{BK04} L. Block and J. Keesling, {\it A characcterization of adding machine maps}, Top. Appl. 140 (2004), 151-161.

\bibitem{Bl15} A. Blokh, {\it Pointwise-recurrent maps on uniquely arcwise connected locally arcwise connected spaces},
 Proc. Amer. Math. Soc. 143 (2015), 3985-4000.

\bibitem{Botelho88} F. Botelho, {\it Rotation sets of maps of the annulus}, Pacific J. Math. 133 (1988), 251-266.

\bibitem{Bron79}I. U. Bronstein, {\it Extensions of minimal transformation groups}, Sitjthofff and Noordhoff, 1979.

\bibitem{C13} C. Carath\'eodory, {\it \"Uber die Begrenzung einfach zusammenh\"angender Gebiete},
Math. Ann. 73 (1913),  323-370.

\bibitem{Ca13} C. Carath\'eodory, {\it \"Uber die gegenseitige Beziehung der R\"ander bei der konformen Abbildung des Inneren einer Jordanschen Kurve auf einen Kreis}, Math. Ann. 73 (1913),  305-320.

\bibitem{CL51}  M. L. Cartwright and J.E. Littlewood, {\it Some fixed point theorems}, with an
 appendix by H. D. Ursell, Ann. of Math. 54 (1951), 1-37.

\bibitem{El58} R. Ellis, {\it Distal transformation Groups}, Pacific J. Math. 8 (1958) 401-405

\bibitem{Fa87} A. Fathi, {\it An orbit closing proof of Brouwer's
lemma on translation arcs}, L'enseignement Mathematique
33 (1987), 315-322.

\bibitem{Fo65} N. E. Foland, {\it A characterization of the almost periodic homeomorphisms on the closed 2-cell}, Proc.
 Amer. Math. Soc. 16 (1965), 1031-1034.

\bibitem{Fr88} J. Franks,  {\it Generalizations of the Poincare-Birkhoff theorem}, Annals of Mathematics, 128 (1988), 139-151.

\bibitem{Fu63} H. Furstenberg, {\it The Structure of distal flows}, Amer. J. Math. 85 (1963), 477-515.

\bibitem{Handel90} M. Handel, {\it The rotation set of homeomorphism of the annulus is closed}, Commun. Math. Phys. 127, 339-349 (1990).

\bibitem{HJ19} T. Hauser and T. J\"ager, {\it Monotonicity of maximal equicontinuos factors and an application to toral flows}, Proc. Amer. Math. Soc. 147 (2019), 4539-4554.

\bibitem{HO16} L. C. Hoehn, L. G. Oversteegen, {\it A complete classification of homogeneous plane continua},
Acta Math. 216 (2016), 177-216.

 \bibitem{KP98} B. Kolev and M. P\'erou\`eme, {\it Recurrent surface homeomorphisms}, Math. Proc. Cambridge Philos. Soc.
 124 (1998), 161-168.

\bibitem{KLN15} A. Koropecki, P. Le Calvez, and M. Nassiri, {\it Prime ends rotation numbers and periodic points},
Duke Math. J. 164 (2015), 403-472.

\bibitem{K66} K. Kuratowski, {\it Topology}, Vol. II, Academic Press, New York, N.Y., 1968.

\bibitem{L06} P. Le Calvez, {\it From Brouwer theory to the study of homeomorphisms of surfaces},
European Mathematical Society (EMS), Z\"urich, 2006, 77-98.

\bibitem{Mai05} J. H. Mai, {\it Pointwise-recurrent graph maps}, Ergodic Theory Dynam. Systems 25 (2005), 629-
637.

\bibitem{MY02} J. H. Mai and X. D. Ye, {\it The stucture of pointwise recurrent maps having the pseudo orbit tracing property,}
 Nagoya Math. J. 166 (2002), 83-92.

\bibitem{M79}   R. Ma\~n\'e, {\it Expansive homeomorphisms and topological dimension}, Trans. Amer. Math. Soc. 252
 (1979), 313-319.

\bibitem{MZ89} M. Misiurewicz and K. Ziemian, {\it Rotation sets for maps of tori}, J. London Math. Soc. (2) 40, 490-506 (1989).


\bibitem{M37} D. Montgomery, {\it Pointwise periodic homeomorphisms,} Amer. J. Math. 59 (1937),  118-120.

\bibitem{MZ55} D. Montgomery and L. Zippin, {\it Topological transformation groups}, Interscience Pub
lishers, New York-London, 1955.

\bibitem{N92} S. Nadler, {\it Continuum theory-An introduction}, Monographs and Textbooks in Pure
 and Applied Mathematics, 158. Marcel Dekker, Inc., New York, 1992.

 \bibitem{Na13}  I. Naghmouchi, {\it Pointwise recurrent dendrite maps}, Ergod. Theory Dyn. Syst. 33 (2013), 1115-1123.

 \bibitem{OT90} L. G. Oversteegen and E. D. Tymchatyn, {\it Recurrent homeomorphismes on $\mathbb R^2$ are periodic,} Proc. Amer.
 Math. Soc. 110 (1990), 1083-1088.

\bibitem{Rees77} M. Rees, {\it On the structure of minimal distal transformation groups with manifolds as phase spaces}, thesis, University of Warwick, Coventry, England, 1977.

 \bibitem{R78}  G. X. Ritter, {\it A characterization of almost periodic homeomorphisms on the 2-sphere and the annulus},
 General Topology Appl. 9 (1978),  185-191.

\bibitem{SXY23} E. H. Shi, H. Xu, and Z. Q. Yu, {\it The structure of pointwise recurrent expansive homeomorphisms}, Ergodic
 Theory Dynam. Systems 43 (2023),  3448-3459.


\bibitem{W82} P. Walters, {\it An introduction to ergodic theory}, GTM 79, Springer-Verlag, 1982.


\end{thebibliography}
\end{document}